\documentclass[envcountsect,runningheads]{llncs}

\usepackage[ruled]{algorithm2e} 

\usepackage{xcolor}
\usepackage[utf8]{inputenc} 
\usepackage[T1]{fontenc}    
\usepackage{hyperref}       
\usepackage{url}            
\usepackage{amsfonts}       
\usepackage{amsmath}
\usepackage{subcaption}
\usepackage{amssymb}

\usepackage{graphicx}

\spnewtheorem{prop}{Proposition}[section]{\bfseries}{\itshape}

\newcommand{\mS}{{{\mathcal M}}} 
\newcommand{\setR}{{\mathcal R}}
\newcommand{\setC}{{\mathcal C}}

\newcommand{\Ad}{{\mathcal A}_d}
\newcommand{\As}{{\mathcal A}_s}

\newcommand{\xs}{x^*}
\newcommand{\LsC}{L^*_C}
\newcommand{\lsd}{l^*_d}

\newcommand{\Bbl}{\Big ( }
\newcommand{\Bbr}{\Big ) }
\newcommand{\Bsbl}{\Big [ }
\newcommand{\Bsbr}{\Big ] }

\newcommand{\dl}{\delta}
\newcommand{\ddl}{\delta_{l,1}}
\newcommand{\ddr}{\delta_{l,2}}
\newcommand{\dsl}{\delta_{s,1}}
\newcommand{\dsr}{\delta_{s,2}}

\DeclareMathOperator*{\argmax}{arg\,max}

\newtheorem{assumption}{Assumption}

\title{Stable Community Structures and Social Exclusion}

\author{Boxuan Li\inst{1}\and
Martin Carrington\inst{2}\and
Peter Marbach\inst{2}}
\authorrunning{B. Li  et al.}
\institute{The University of Hong Kong (HKU), Hong Kong\\
\email{liboxuan@connect.hku.hk}\\  
  \and
  University of Toronto, Canada\\
\email{\{carrington,marbach\}@cs.toronto.edu}}

\begin{document}
\maketitle

\begin{abstract}
In this paper we study social exclusion in social (information) networks using a game-theoretic approach, and study the stability of a certain class community structures that are a Nash equilibrium. The main result of our analysis shows that all stable community structures (Nash equilibria) in this class are community structures under which some agents are socially excluded, and do not belong to any of the communities.  This result is quite striking as it suggests that social exclusion might be the ``norm'' (an expected outcome) in social networks, rather than an anomaly.
\end{abstract}

\section{Introduction}
\label{section:introduction}
Social exclusion has been recognized as an important aspect in understanding social networks~\cite{who}. Roughly, social exclusion can be characterized by a distribution of goods and services that excludes certain (groups of) individuals from having any access at all to these goods and services. Or the goods and services that they have access to do not match their needs and as such do not provide any ``benefit'. 
Two key questions for understanding social exclusion are  a) ``what is the process that leads to social exclusion?'', and b) ``what (social) policies can be put in place in order to re-integrate excluded (marginalized) individuals back into a community (society)?''. In this paper we focus on the first question and  study the process that leads to  social exclusion. While concentrating on this question, we believe that the models and insights into the process of how social exclusion occurs can potentially also be used to analyze and design (social) policies to counteract social exclusion.

The question  of how social exclusion occurs has been studied in the literature where the focus has been on studying “who  is  doing  the excluding?”~\cite{aktinson}. In this paper we take a different approach and rather than focusing on the question ``who is doing the exclusion'', we study whether social exclusion might be in fact a structural property of communities in social networks. To do this, we use a game-theoretic approach to model the interaction among agents in a social network that has been studied by Carrington and Marbach~\cite{carrington2019community}. Using this game-theoretic framework, we study  a) which community structures (Nash equilibria) that emerge in a social network are stable, and b) whether there exist stable community structures under which some agents are  excluded (marginalized) and do not belong to any community. 

The key result that we obtain through our analysis is that the only community structures (Nash equilibria) that are stable are community structures with social exclusion. That is, the only community structures that are stable are structures under which some agents are marginalized and do not belong to any community. This is a striking result that suggests that social exclusion is indeed a ``reality'' (inevitable) in social networks. In this sense, the results obtained by our analysis suggest a new, and fundamentally different, understanding of social exclusion. That is rather than focusing on the question ``who is doing the exclusion'', one should focus on how do we deal (as a society) with social exclusion as a systematic property of social networks? We discuss this in more details in Section~\ref{section:conclusion}.


The rest of the paper is organized as follows. In Section \ref{section:background} we describe the model we adopt to study community networks. In Section \ref{sec:model} we introduce the perturbation model that we use to study the stability of Nash equilibria. In Section \ref{section:analysis} we present our main results.


\section{Related Work}
\label{section:related_work}
Studying social exclusion has received considerable attention in  economic, social and health sciences; we refer to~\cite{who} for an overview of this vast literature. However studying social exclusion using a formal, model-based approach, has obtained much less attention. The only existing literature that uses a formal model-based approach to study social exclusion that we are aware of  is the paper by Carrington and Marbach~\cite{carrington2019community}.

Using a game-theoretic approach, the research in~\cite{carrington2019community} formally studies whether  in social (information) networks there exist Nash equilibria under which some agents are marginalized, and not included in any of the communities. To do that, the work in~\cite{carrington2019community} considers a particular type of social networks  where agents (individuals) share/exchange information. Each agent chooses to join communities that maximizes its own utility obtained from content obtained, and shared, within the community. The analysis in~\cite{carrington2019community} shows that there exists a class of Nash equilibria for the resulting game under which some agents are excluded (marginalized) from the community structure, i.e. they do not belong to (join) any of the communities. The reason for this is that these agents would have a negative utility in all of the communities that exist in the Nash equilibrium. These agents then have the choice to either join a community where the utility they receive is negative, or not join any community at all (and obtain a utility of zero).  In this situation agents are better off not joining any community, and they become marginalized.

While the analysis in~\cite{carrington2019community} shows that Nash equilibria with marginalized agents do exist, it does not address the question whether these Nash equilibria are indeed likely to occur and persist in a social network. Note that communities  in social network constantly change (are perturbed)  as some agents  leave, and other agents join, the community.  For this situation, we are interested in studying whether a Nash equilibrium is stable in the sense that it is robust to small changes (perturbations). Stable Nash equilibria are likely to persist and hence to be observed in social networks. On the other hand, Nash equilibria that are not stable will eventually vanish, and as a result are not likely to occur.  In the following  we are interested in whether Nash equilibria with marginalized agents are robust to perturbations, and hence likely to occur and persist in social networks.

In our analysis we use the concept of a stable Nash equilibrium that we formally define in Section~\ref{sec:model}. Roughly,  we define a stable Nash equilibrium as a Nash equilibrium that is robust to (local) perturbations. There is an extensive literature in game-theory that studies the properties of Nash equilibria under perturbations.
These perturbations could either be perturbations to the agents' value function, or to the agents' (players') strategies~\cite{jackson,balcan}. The approach that we consider is most closely related to stochastic fictitious play~\cite{fudenberg,hofbauer,candogan} where each player's payoffs are perturbed in each period by
random shocks. The convergence results of stochastic fictitious play has been analyzed for games  with  an  interior  evolutionary stable strategy (ESS),  zero  sum  games,  potential  games, near potential games and  supermodular  games. These results are obtained using techniques from stochastic approximation theory that show that one can characterize the perturbed best response dynamic of stochastic fictitious games by a differential equation defined by the expected motion of the stochastic process. We use the same approach in this paper where we model the perturbed best response dynamics of the agents by a differential equation as given in Section~\ref{sec:model}.


\section{Background}
\label{section:background}

In this section we describe the model and results of~\cite{carrington2019community} that we use for our analysis. The model considers the situation where agents produce (generate)  and consume (obtain) content in a social (information) network. Agents can  form communities in order to share/exchange content more efficiently, and obtain a certain utility from joining a given community. Using a game-theoretic framework, the community structure that emerges is characterized by a Nash equilibrium.

More precisely, the model in~\cite{carrington2019community} is given as follows. Assume that each content item that is being produced in the social (information) network can be associated with a particular topic, or content type. Furthermore assume that there exists  a structure that relates different content topics with each other. In particular, assume there exists a measure of ``closeness'' between content topics that characterizes how strongly related two content topics are. To model this situation, the topic of a content item is given by a point $x$ in a metric space, and the closeness between two content topics $x,x' \in \mS$ is then given by the distance measure $d(x,x')$, $x,x' \in \mS$, for the metric space $\mS$.  

The set of agents in the network, and agents' interests as well as agents' ability to produce content are then given as follows. Assume that there is a set $\Ad$ of agents that consume content, and a set $\As$ of agents that produce content, where the subscripts stand for ``demand" and ``supply''. Furthermore, associate with each agent that consumes content a center of interest $y \in \mS$, i.e. the center of interest $y$ of an agent is the topic that the agent is most interested in. The interest in content topic $x$ of an agent with center of interest $y$ is given by
\begin{equation}\label{eq:p}
p(x|y) = f(d(x,y)), \qquad x,y \in \mS,
\end{equation}
where $d(x,y)$ is the distance between the center of interest $y$ and topic $x$, and $f:[0,\infty) \mapsto [0,1]$ is a non-increasing function. The interpretation of the function $p(x|y)$ is as follows: when an agent with center of interest $y$ consumes (reads) a  content item on topic $x$, then it finds it interesting with probability $p(x|y)$ as given by Eq.~\eqref{eq:p}.
As the function $f$ is non-increasing, this model captures the intuition that the agent is more interested in topics that are close to its center of interest $y$. 

Similarly, given an agent that produces content, the center of interest $y$ of the agent is the topic for which  the agent is most adept at producing content.  The ability of the agent to produce content on topic $x \in \mS$ is given by
\begin{equation}\label{eq:q}
q(x|y) = g(d(x,y)),
\end{equation}
where $g:[0,\infty) \mapsto [0,1]$ is a non-increasing function.

  In the following we identify an agent by its center of interest $y \in \mS$, i.e. agent $y$ is the agent with center of interest $y$. As a result we have that $\Ad \subseteq \mS$ and $\As \subseteq \mS$.

\subsection{Community $C = (C_d, C_s)$}\label{ssec:information_community}

A  community $C = (C_d, C_s)$ consists of a set of agents that consume content $C_d \subseteq \Ad$ and a set of agents that produce content $C_s \subseteq \As$.
Let $\beta_C(x|y)$ be the rate at which agent $y \in C_s$ generates content items on topic $x$ in community $C$. Let $\alpha_{C} (y)$ be the fraction of content produced in community $C$ that  agent $y \in C_d$ consumes. To define the utility for content consumption and production, assume that when an agent consumes a single content item, it receives a reward equal to 1 if the content item is of interest and relevant, and pays a cost of $c>0$ for consuming the item. The cost $c$ captures the cost in time (energy) to read/consume a content item.
Using this reward and cost structure, the utility rate (``reward minus cost") for content consumption of agent $y \in C_d$ is given by 
\begin{align}
U_C^{(d)}(y) &= 
 \alpha_{C}(y) \int_{x \in \mS} \Bsbl Q_C(x) p(x|y) - \beta_{C}(x)c \Bsbr dx, \nonumber
\end{align}
where 
$$ Q_C(x) = \int_{y \in C_s} \beta_C(x|y) q(x|y) dy, \;  \mbox{ and } \; \beta_C(x) = \int_{y \in C_s} \beta_C(x|y) dy.
$$
Similarly, the utility rate for content production of agent $y \in C_s$ is given by
$$U_C^{(s)}(y) = 
 \int_{x \in \mS} \beta_{C}(x|y)  \Bsbl q(x|y)P_C(x) - \alpha_{C} c \Bsbr dx  ,$$ 
where
$$P_C(x) = \int_{y \in C_d} \alpha_C(y) p(x|y) dy, \; \mbox{ and } \;
\alpha_C = \int_{z \in C_d} \alpha_C(z) dz.
$$
The utility rate for content production captures how ``valuable'' the content produced by agent $y$ is for the set of content consuming agents $C_d$ in the community $C$~\cite{carrington2019community}.

\subsection{Community Structure and Nash Equilibrium}
\label{section:cs}
A community structure defines   how agents organize themselves into communities, where in each community agents produce and consume content as described in the previous subsection.

A community structure is given by a triplet 
 $ (\mathcal{C}, \{\alpha_\mathcal{C}(y) \}_{y \in \Ad}, \{\beta_\mathcal{C}(\cdot|y) \}_{y \in \As} )$, 
where  $\mathcal{C}$ is the set of communities $\mathcal{C}$ that exist in  the structure, and 
$$
 \alpha_\mathcal{C}(y) = \{\alpha_{C}(y)\}_{C \in \mathcal{C}}, y \in \Ad, 
  \ \ \mathrm{and} \ \ 
  \beta_\mathcal{C}(\cdot|y) = \{\beta_{C}(\cdot|y)\}_{C \in \mathcal{C}}, y \in \As,  
$$
  are the consumption fractions and production rates, respectively, that agents allocate to the different communities $C \in \mathcal{C}$.

Assume that the total consumption fractions and production rates of each agent are bounded by $0 < E_p  \leq 1$, and $E_q >0$, respectively, i.e. we have that
$$
\Vert\alpha_\mathcal{C}(y)  \Vert= 
\sum_{C \in \mathcal{C}} \alpha_{C} (y)  \leq E_p , \;y \in \Ad,
\; \mbox{ and } \;
\Vert\beta_\mathcal{C}(y)  \Vert=
 \sum_{C \in \mathcal{C}} \Vert \beta_{C} (\cdot|y)  \Vert  \leq E_q, \; y \in \As,
$$
where 
$   \Vert \beta_{C} (\cdot|y)  \Vert = \int_{x \in \mS} \beta_C(x|y) dx $.

To analyze the interaction among agents, assume that agents join communities in order to maximize their own utility rates. That is,  agents join communities, and choose allocations $\alpha_\mathcal{C}(y)$ and $\beta_\mathcal{C}(\cdot|y)$, in order  to maximize their own  consumption, and production utility rates, respectively.

A Nash equilibrium is then given by a  community structure \\ $(\mathcal{C}^*, \{\alpha_\mathcal{C}^*(y) \}_{y \in \Ad}, \{\beta_\mathcal{C}^*(\cdot|y) \}_{y \in \As  } )$ such that for all agents $y \in \Ad$ we have that 
$$\alpha_\mathcal{C}^*(y) = \argmax_{\alpha_\mathcal{C}(y): \Vert\alpha_\mathcal{C}(y)  \Vert \leq E_p} \sum_{C \in \mathcal{C}}  U_C^{(d)} (y), $$
and for all agents $y \in \As$, we have that
$$ \beta_\mathcal{C}^*(\cdot|y) = 
\argmax_{\beta_\mathcal{C}(\cdot|y): \Vert\beta_\mathcal{C}(y)  \Vert \leq E_q} 
\sum_{C \in \mathcal{C}}  
U_C^{(s)} (y).$$

\subsection{Community Structure $\mathcal{C}(L_C,l_d)$}
The above model was analyzed in~\cite{carrington2019community} for the case of a specific metric space, and a specific family of information communities. More precisely, the analysis in~\cite{carrington2019community} considers the one-dimensional metric space given by an interval  $\setR = [-L, L) \subset \mathbb{R}$, $L >0 $, with the torus metric, i.e. the distance between two points $x,y \in \setR$ is given by 
$$d(x,y) = ||x-y|| = \min \{ |x-y|, 2L - |x-y|\},$$
  where $| x |$ is the absolute value of 
  $x \in \mathbb{R}$. The metric space $\setR$ with the torus metric is the simplest (non-trivial) one-dimensional metric space for the analysis of community structures in information networks. The reason for this is that the torus metric is ``symmetric'' and does not have any ``border effects'', which simplifies the analysis.
  
  Furthermore, the analysis in~\cite{carrington2019community} assumes that
$\Ad = \As = \setR$,
i.e. for each content topic $x \in \setR$ there exists an agent in $\Ad$ who is most interested in content of type $x$, and  there exists an agent in $\As$ who is most adept at producing  content of type $x$.

In addition, the analysis in~\cite{carrington2019community} considers a particular family  of community structures $\setC(L_C,l_d)$, $L_C >0$ and $0 < l_d \leq L_C$, given as follows.


Let $N \geq 2$ be a given integer. Furthermore, let $L_C = \frac{L}{N}$,
where $L$ is the half-length of the metric space $\setR=[-L,L)$, 
and let $l_d$ be such that $ 0 < l_d \leq L_C$.
Finally, let $ \{m_k\}_{k = 1}^N$ be a set of $N$ evenly spaced points on the metric space $\setR=[-L,L)$ given by
$ m_{k+1} =  m_1 + 2L_Ck$, $k = 1,...,N-1$.

Given $L_C$, $l_d$, and $m_k$, $k=1,...,N$, as defined above, the set of communities  $\mathcal{C} =  \{C^k = (C_d^k,C_s^k)\}_{k = 1}^N$ of the structure $\setC(L_C,l_d)$ is then given by the intervals 
$$C_{d}^k =  [m_k - l_d,m_k + l_d) \mbox{ and }\;
C_{s}^k =  [m_k - L_C,m_k + L_C).$$

Furthermore, the allocations $\{\alpha_\mathcal{C}(y) \}_{y \in \mathcal{R}}$ and 
$\{\beta_\mathcal{C}(\cdot|y) \}_{y \in \mathcal{R}  }$ of the community structure $\mathcal{C}(L_C,l_d)$ are given by 
\begin{equation}\label{eq:a_C}
\alpha_{C^k}(y) 
=
\begin{cases}
E_p & y \in C^k_d \\
0 & \text{otherwise}
\end{cases},
 \qquad \qquad k = 1,...,N,
\end{equation}
and 
\begin{equation}\label{eq:b_C}
\beta_{C^k}(\cdot|y) 
=
\begin{cases}
E_q \delta(x-\xs(y)) & y \in C^k_s \\
0 & \text{otherwise}
\end{cases},
 \qquad \qquad k = 1,...,N,
\end{equation}
where 
\[
\xs(y) = \argmax_{x \in \mathcal{R}} q(x|y) P_{C^k}(x).
\]

Note that for $l_d = L_C$, the community structure $\setC(L_C,l_d) = \setC(L_C,L_C)$, and all agents belong to at least one community in $\setC(L_C,L_C)$.
On the other hand if we have that $l_d < L_C$ then community structure has ``gaps'', and there are agents that do not belong to any community.
In particular, the content consuming agents in the sets
$$D^k = [m_k + l_d,m_{k+1}-l_d), \qquad k=1,...,N-1,$$
and $D^N =  [m_N + l_d,m_1-l_d)$,  do not belong to any community in $\setC(L_C,l_d)$. This means that these agents are marginalized, and excluded from the community structure.


\subsection{Nash Equilibria  $\setC(L^*_C,l^*_d)$}
  
To analyze the existence of  Nash equilibria within the family $\setC(L_C,l_d)$ of community structures as defined in the previous section, ~\cite{carrington2019community} made the following assumptions for the functions $f$ and $g$  that are used in Eq.~\eqref{eq:p} and Eq.~\eqref{eq:q}.
\begin{assumption}
\label{assumption:simple_fg}
The function $f: [0,\infty) \mapsto [0,1]$ is given by 
$f(x) = \mathrm{max} \{ 0, f_0 - ax \}$, where $f_0 \in (0,1]$ and  $a >0$.
The function $g: [0,\infty) \mapsto [0,1] $ is given by
$g(x) = g_0$, where $g_0 \in (0,1]$.
Furthermore, we have that $f_0 g_0 > c$.
\end{assumption}
The condition in Assumption~\ref{assumption:simple_fg} that $f_0 g_0 > c$ is a necessary condition for a Nash equilibrium to exist, i.e. if this condition is not true, then there does not exist a Nash equilibrium~\cite{carrington2019community}. 
Under the above assumptions, the following two results regarding the existence, and properties of,  Nash equilibria within the family $\setC(L_C,l_d)$ of community structures were obtained in~\cite{carrington2019community}.

The first result states that there always exists a Nash equilibrium with $l_d = L_C$  within the family $\setC(L_C,l_d)$ of community structures as defined above.

\begin{prop}\label{prop:NE_LC}
  Let the functions $f$ and $g$ be as given in Assumption~\ref{assumption:simple_fg}. Then the  community structure $\mathcal{C}(L^*_C,\LsC)$  is a Nash equilibrium if, and only if,
$$\LsC \leq  \frac{f_0}{a} - \frac{c}{a g_0}.$$
\end{prop}

Proposition~\ref{prop:NE_LC} states that there always exists a Nash equilibrium under which no agents are marginalized. 

The next result provides characterization of a Nash equilibrium with marginalized agents.
\begin{prop}
\label{prop:NE_ld}
Let the functions $f$ and $g$ be as given in Assumption~\ref{assumption:simple_fg}. Then the community structure $\mathcal{C}(L^*_C,l^*_d)$ with 
$$0 < l_d^* < L_C^*, \ \ \ \mathrm{ and }  \ \ \ L_C^* = \frac{L}{N}$$
where $N \geq 2$ is an integer, 
is a Nash equilibrium
if, and only if,
$$l_d^* =  \frac{f_0}{a} - \frac{c}{a g_0}.$$
\end{prop}
Note that Proposition~\ref{prop:NE_ld} that there always exists a  community structure $\mathcal{C}(L^*_C,l^*_d)$ that is a Nash equilibrium with marginalized agents if we have that
$$L > 2 \left ( \frac{f_0}{a} - \frac{c}{a g_0} \right ),$$
i.e. if the content space $(-L,L]$ is large enough.



\section{Stable  and Neutral-Stable Nash Equilibria  $\mathcal{C}(L^*_C,l^*_d)$}\label{sec:model}

The results from~\cite{carrington2019community} presented in the previous section show that there always exists a Nash equilibrium  $\mathcal{C}(L^*_C,l^*_d=L^*_C)$ under which no agents are marginalized. Furthermore, it shows that if the content space $(-L,L]$ is large enough then there always exists a Nash equilibrium  $\mathcal{C}(L^*_C,l^*_d)$, $0 < l^*_d < L^*_C$  under which some agents are marginalized.

The goal of this paper is to characterize which of the Nash equilibria  $\mathcal{C}(L^*_C,l^*_d)$ obtained in the previous section are stable, i.e. robust to small perturbations to the community structure. This question is motivated by the following observation. As over time agents may join, or leave, communities in a social network, community structure is not static but changes over time. For this situation, we are interested in studying whether a community structure is stable in the sense that it is robust to small changes (perturbations) to the structure.

In order to study this question, we use the following approach. For a given  Nash equilibrium $\mathcal{C}(L^*_C,l^*_d)$, we consider two adjacent communities $C^1 = (C^1_d,C^1_s)$ and  $C^2 = (C^2_d,C^2_s)$. Without loss of generality we assume that the two communities are given by
\begin{equation}\label{eq:C_1_d}
  C^1_d = [-L_C - l^*_d, -L_C + l^*_d)
\end{equation}    
and 
\begin{equation}\label{eq:C_2_d}
  C^2_d = [L_C - l^*_d, L_C + l^*_d),
\end{equation}
as well as
\begin{equation}\label{eq:C_1_s}
  C^1_s = [-2L_C, 0)
\end{equation}
and
\begin{equation}\label{eq:C_2_s}
  C^2_s = [0,2L_C).
\end{equation}
Note that by the definition of the community structure $\mathcal{C}(L_C,l^*_d)$ given in Section~\ref{sec:model}, there always exist at least two communities in a Nash equilibrium $\mathcal{C}(L^*_C,l^*_d)$.

We then perturb the ``boundary'' between the two communities by a small amount, and study the dynamics of the boundaries between the two communities after this perturbation. In particular we study whether over time the perturbation vanishes and the original community structure is again obtained. If the perturbation vanishes over time, then this suggests that the Nash equilibrium is stable in the sense that the Nash equilibrium  is resistant to (local) perturbations. On the other hand if the   communities  are no restored after the perturbation, then this suggests that the Nash equilibrium is not  stable.

More precisely, we consider the following perturbation model to study the stability of a Nash equilibrium.
For $t=0$ we set
$$C^1_d(t=0) = [-L_C - l^*_d, -L_C + l^*_d + \delta_{dl}(0) )$$
and
$$C^2_d(t=0) = [L_C - l^*_d + \delta_{dr}(0), L_C + l^*_d),$$
as well as
$$C^1_s(t=0) = [-2L_C, \delta_{sl}(0)) 
\; \mbox{ and } \;
C^2_{s}(t=0) = [\delta_{sr}(0), 2L_C),$$   
where
$$ \delta_{dl}(0) = \delta_{dr}(0), \quad \mbox{ if } l^*_d = L_C$$
and 
$$\delta_{dl}(0) < 2L_C - 2l^*_d +  \delta_{dr}(0), \quad \mbox{ if } l^*_d < L_C,$$
as well as
$\delta_{sl}(0) = \delta_{sr}(0)$.

We analyze the dynamics of the boundaries between the two communities after the initial perturbation, i.e. we characterize the trajectory of the perturbation  $\delta_{dl}(t)$, $\delta_{dr}(t)$, $\delta_{sl}(t)$, and  $\delta_{sr}(t)$, over time $t \geq 0$. The intuition behind the dynamics of the boundaries is as follows.  Note that the community boundaries represent the agents at the border between the two communities. We refer to these agents as the border agents. We then assume that a border agent will  join the community which provides the highest utility, given that the highest utility is non-negative. If the highest utility is negative, then the border agent will leave the community and not join any community (and obtain a utility equal to 0).

Using this intuition, we then assume that the rate at which the boundaries move is given by the difference between the utilities in the two communities. That is, the higher the difference of the two utilities, the higher the rate (the faster) with which  border agents move from one community to the other. The corresponding differential equations for the dynamics of the boundaries for $t \geq 0$ are then given by
\begin{equation}
\label{eqn:dl_deri}
\frac{ d \delta_{dl}(t)}{dt}=  U_{C_1(t)}^{(d)} ( -L_C + l^*_d + \delta_{dl}(t) ) - \max \left \{ 0,  U_{C_2(t)}^{(d)} ( -L_C + l^*_d + \delta_{dl}(t)) \right \}
\end{equation}
and
\begin{equation}
\label{eqn:dr_deri}
\frac{ d \delta_{dr}(t)}{dt}=  \max \left \{ 0, U_{C_1(t)}^{(d)} ( L_C - l^*_d + \delta_{dr}(t) ) \right \} - U_{C_2(t)}^{(d)} ( L_C - l^*_d + \delta_{dr}(t) )
\end{equation}
Similarly, we have that
\begin{equation}
\label{eqn:sl_deri}
\frac{ d \delta_{sl}(t)}{dt}=  U_{C_1(t)}^{(s)} ( -L_C + l^*_d + \delta_{sl}(t) ) - \max \left \{ 0,  U_{C_2(t)}^{(s)} ( -L_C + l^*_d + \delta_{sl}(t) ) \right \}
\end{equation}
and
\begin{equation}
\label{eqn:sr_deri}
\frac{ d \delta_{sr}(t)}{dt}=  \max \left \{ 0, U_{C_1(t)}^{(s)} ( L_C - l^*_d + \delta_{sr}(t) ) \right \} - U_{C_2(t)}^{(s)} ( L_C - l^*_d + \delta_{sr}(t) ).
\end{equation}

Note that in the above differential equation the border agent will never move to a community that provides a negative utility.

Using this model, we say that the Nash equilibrium is stable if the following is true.

\begin{definition}\label{def:stable}
Let  $\mathcal{C}(L^*_C,l^*_d)$ be a given Nash equilibrium, and let  $C_1 = (C^1_d,C^1_s)$ and  $C_2 = (C^2_d,C^2_s)$  be two communities in  $\mathcal{C}(L^*_C,l^*_d)$ as given by Eq.~\eqref{eq:C_1_d}~-~\eqref{eq:C_2_s}. 

 We say that the Nash equilibrium $\mathcal{C}(L^*_C,l^*_d)$ is stable if there exists a $\delta > 0$ such that for
$$ 0 < \delta_{dl}(t=0),  \delta_{dr}(t=0), \delta_{sl}(t=0), \delta_{sr}(t=0) \leq \delta$$
we have for the differential equations given by Eq.~\eqref{eqn:dl_deri}~-~\eqref{eqn:sr_deri} that
$$  0 \leq \delta_{dl}(t),  \delta_{dr}(t), \delta_{sl}(t), \delta_{sr}(t) \leq \delta, \qquad t \geq 0,$$
and
$$ \lim_{t \to \infty}  \delta_{dl}(t) =  \lim_{t \to \infty}  \delta_{dr}(t) =  \lim_{t \to \infty}  \delta_{sl}(t) =  \lim_{t \to \infty}  \delta_{sr}(t) = 0.$$
\end{definition}
This definition captures the intuition that under a stable Nash equilibrium we have that small perturbations to the community boundaries will vanish (become equal to 0) over time.

In addition, we use a weaker notion of stability to which we refer to as a neutral-stable Nash equilibrium. To define a neutral-stable Nash equilibrium, we use the following notation.

Let  $C_1(t) = (C^1_d(t),C^1_s(t))$ and  $C_2(t) = (C^2_d(t),C^2_s(t))$, be the structure of two communities under the above perturbation model at time $t$.  For an agent $y \in C^1_s(t) \cup C^2_s(t)$, let
$$\xs_1(y,t) = \argmax_{x \in \mathcal{R}} q(x|y) \int_{z \in C^1_d(t)} p(x|z)dz$$
be the optimal content for agent $y$ to produce in community $C_1(t)$, and let
$$\xs_2(y,t) = \argmax_{x \in \mathcal{R}} q(x|y) \int_{z \in C^2_d(t)} p(x|z)dz$$
be the optimal content for agent $y$ to produce in community $C_2(t)$.
Using this definition, let the utilities of agents $y \in C^1_d(t) \cup C^2_d(t)$ be given as follows,
$$U_{C_1(t)}^{(d)}(y) = E_p E_q \int_{z \in C^1_s(t)} \Bsbl q(\xs_1(z,t)|z)p(\xs_1(z,t)|y) - c \Bsbr dz$$
and
$$U_{C_2(t)}^{(d)}(y) = E_p E_q \int_{z \in C^2_s(t)} \Bsbl q(\xs_2(z,t)|z)p(\xs_2(z,t)|y) - c \Bsbr dz.$$
Similarly,  let the utilities of agents $y \in C^1_s(t) \cup C^2_s(t)$ be given as follows,
$$U_{C_1(t)}^{(s)}(y) = E_p E_q \int_{z \in C^1_d(t)} \Bsbl q(\xs_1(y,t)|y)p(\xs_1(y,t)|z) - c \Bsbr dz$$
and
$$U_{C_2(t)}^{(s)}(y) = E_p E_q \int_{z \in C^2_d(t)} \Bsbl q(\xs_2(y,t)|y)p(\xs_2(y,t)|z) - c \Bsbr dz.$$

\begin{definition}\label{def:neutral-stable}
Let  $\mathcal{C}(L^*_C,l^*_d)$ be a given Nash equilibrium, and let  $C_1 = (C^1_d,C^1_s)$ and  $C_2 = (C^2_d,C^2_s)$ be two communities in  $\mathcal{C}(L^*_C,l^*_d)$ as given by Eq.~\eqref{eq:C_1_d}~-~\eqref{eq:C_2_s}. 
  We say that the Nash equilibrium $\mathcal{C}(L^*_C,l^*_d)$ is neutral-stable if one of the following is true for the differential equations given by Eq.~\eqref{eqn:dl_deri}~-~\eqref{eqn:sr_deri}.
\begin{enumerate}  
\item There exists a $\delta > 0$ such that for
$$ 0 < \delta_{dl}(t=0),  \delta_{dr}(t=0), \delta_{sl}(t=0), \delta_{sr}(t=0) \leq \delta$$
we have 
$$  0 \leq \delta_{dl}(t),  \delta_{dr}(t) \leq \delta, \qquad t \geq 0,
\; \mbox{ and }
 \lim_{t \to \infty}  \delta_{dl}(t) =  \lim_{t \to \infty}  \delta_{dr}(t) = 0,$$
as well as
$$ \lim_{t \to \infty} U_{C_1(t)}^{(s)}(y) =  \lim_{t \to \infty}  U_{C_2(t)}^{(s)}(y) > 0, \qquad y \in C^1_s \cup  C^2_s.$$
\item There exists a $\delta > 0$ such that for
$$ 0 < \delta_{dl}(t=0),  \delta_{dr}(t=0), \delta_{sl}(t=0), \delta_{sr}(t=0) \leq \delta$$
we have 
$$  0 \leq \delta_{sl}(t), \delta_{sr}(t) \leq \delta, \qquad t \geq 0,
\; \mbox{ and } \;
 \lim_{t \to \infty}  \delta_{sl}(t) =  \lim_{t \to \infty}  \delta_{sr}(t) = 0,$$
as well as
$$ \lim_{t \to \infty} U_{C_1(t)}^{(d)}(y) =  \lim_{t \to \infty}  U_{C_2(t)}^{(d)}(y) > 0, \qquad y \in C^1_d \cup  C^2_d.$$
\end{enumerate}
\end{definition}
This definition captures the case where either the content producers, or content consumers, are neutral (indifferent) regarding which community to join as they 
obtain the same utility rate in both communities.



\section{Results}
\label{section:analysis}

We obtain the following results for the perturbation model of Section~\ref{sec:model}.
The first result states that (almost) all Nash equilibrium $\mathcal{C}(L^*_C,\lsd = L^*_C)$ as given by Proposition~\ref{prop:NE_LC} under which no agents are marginalized, are not stable.
\begin{proposition}
\label{prop:not_stable}
All Nash equilibria  $\mathcal{C}(L^*_C,\lsd = L^*_C)$ with
$L^*_C <  \left ( \frac{f_0}{a} - \frac{c}{a g_0} \right )$
are neither  stable nor neutral-stable.
\end{proposition}
We provide a proof of Proposition~\ref{prop:not_stable} in Appendix~\ref{prop:not_stable}.

The next result states that all Nash equilibria $\mathcal{C}(L^*_C,l^*_d)$ as given by Proposition~\ref{prop:NE_ld} under which some agents are  marginalized, are neutral-stable. 
\begin{proposition}
\label{prop:stable}
All Nash equilibria  $\mathcal{C}(L^*_C,\lsd)$ with
$l^*_d = \left ( \frac{f_0}{a} - \frac{c}{a g_0} \right ) < L^*_C$
as given by Proposition~\ref{prop:NE_ld}  are neutral-stable.
\end{proposition}
We provide a proof of Proposition~\ref{prop:not_stable} in Appendix~\ref{prop:stable}.

The above results state that Nash equilibria with marginalized agents are (neutral) stable, but  (almost) all Nash equilibria with no marginalized agents are not stable.





\section{Conclusions}
\label{section:conclusion}

We have studied social exclusion in social networks using a game-theoretic framework. In particular, we asked and analyzed  the question whether social exclusion might be a structural property of communities in social networks.
The results obtained under the model considered in this paper show that all (neutral) stable Nash equilibria are Nash equilibria with social exclusion.  This is quite a striking result as it suggests that having marginalized individuals (i.e. social exclusion) is the ``norm'' in social networks, rather than an anomaly or exception. In this sense, the obtained results provide a new understanding of social exclusion, where rather than asking ``who is doing the exclusion'' the correct question to ask is how do we deal with social exclusion as a systematic property of social networks (society)?

We note that the results in this paper were obtained under a particular game-theoretic model. In that sense, rather than providing a definite ``proof'' that having marginalized individuals (i.e. social exclusion) is the ``norm'' in social networks, rather than an anomaly or exception, the results that we obtain is a first evidence that suggests that this might be the case. Given the potential impact and implications of this result in our understanding of social exclusion, important and interesting follow-up research is to verify whether these results a)  are true under more general models than the one considered in this paper,  and
b) can be verified/observed in real-life case studies?
We discuss possible directions to study these two research questions in more details below.  

The results obtained in this paper were for a particular case where the functions $f$ and $g$ are as given by Assumption~\ref{assumption:simple_fg}. A natural question to ask is whether these results extend to more general functions $f$ and $g$.  An extension of the formal analysis to more general functions $f$ and $g$ seems challenging and it is not clear whether it can be done. As a result, one might need to resort to numerical case studies to investigate this question. We carried out such numerical case studies for more general functions, and the initial results that we obtained suggest that the results indeed carry over to more general settings.

An interesting and important aspect of the results obtained in this paper is whether they indeed provide the correct insight into the process of how social exclusion occurs in real-life social networks and communities. One potential avenue for exploring this question could be using the effect of globalization on social exclusion. This topic has been extensively studied and documented~\cite{who}. In particular, the work by Beall~\cite{beall} provides a concrete study of the influence of globalization on local communities in Pakistan and South Africa. An interesting question is whether the empirical results in~\cite{beall} could be explained by, and match, the model-based results obtained in this paper.

Finally, an interesting direction for future research is to study whether the models used for the analysis in this paper  could  be used to design (social) policies to re-integrate marginalized individuals into a community/society.

\newpage

\bibliographystyle{splncs04}
\bibliography{paper}


\newpage
\appendix
\section{Properties of the Utilities for the Communities $C_1(t)$ and $C_2(t)$}
\label{appendix:utilities}
In this appendix we provide additional properties  of a Nash equilibrium  $\mathcal{C}(L^*_C,l^*_d)$ as given by Proposition~\ref{prop:NE_LC} and Proposition~\ref{prop:NE_ld}, as well as the utility functions of agents in the communities $C_1(t)$ and $C_2(t)$, $t \geq 0$,  as defined in Section~\ref{sec:model}.

Let $\mathcal{C}(L^*_C,l^*_d)$ be a Nash equilibrium as given by Proposition~\ref{prop:NE_LC} and Proposition~\ref{prop:NE_ld}, and let $C = (C_d, C_s)$ be  a community under this Nash equilibrium. From Section~\ref{section:background} we then have that the consumption utility of an agent  $y \in C_d$ is  given by
\begin{equation}
U_C^{(d)} (y) = E_p E_q  \int_{z \in C_s} p(\xs(z)|y)q(\xs(z)|z) dz 
- E_p E_q c \int_{z \in C_s} dz
\label{prod_cons_intervals_consumption}
\end{equation}
and the production utility of an agent $y \in C_s$ is given by
\begin{equation}
U_C^{(s)}(y) = E_p E_q q(\xs(y)|y) \int_{z \in C_d} p(\xs(y)|z)dz - E_q E_q c \int_{z \in C_d}dz
\label{prod_cons_intervals_production},
\end{equation}
where
$$ \xs(y) = \argmax_{x \in \mathcal{R}} q(x|y) \int_{z \in C_d} p(x|z) dz.$$

The following result from~\cite{carrington2019community} characterizes the content $\xs(y)$ that an agent $y$ produces under a  Nash equilibrium as given by Proposition~\ref{prop:NE_LC} and Proposition~\ref{prop:NE_ld}.

\begin{lemma}\label{lemma:xs}
Let $\mathcal{C}(L^*_C,l^*_d)$ be a Nash equilibrium as given by Proposition~\ref{prop:NE_LC} and Proposition~\ref{prop:NE_ld}, and let $C = (C_d, C_s)$ given by
$$C_d = [y_0-l^*_d,y_0+l^*_d)$$
and
$$C_s = [y_0-L^*_C,y_0+L^*_C),$$
be  a community under this Nash equilibrium. Then we have for agent $y \in C_s$ that
$$\xs(y) = y_0 = \argmax_{x \in \mathcal{R}} q(x|y) \int_{z \in C_d} p(x|z) dz.$$
\end{lemma}
Lemma~\ref{lemma:xs} states that it is optimal for an agent $y \in C_s$ to produce content at the center of interest $y_0$ of the agent set $C_d$. 
Using the same argument as given in in~\cite{carrington2019community} to prove Lemma~\ref{lemma:xs}, we obtain the following result.

\begin{lemma}\label{lemma:xs_t}
Let              
$$C^1_d(t) = [-L^*_C - l^*_d, -L^*_C + l^*_d + \delta_{dl}(t) )$$
and
$$C^2_d(t) = [L^*_C - l^*_d + \delta_{dr}(t), L^*_C + l^*_d),$$
as well as
$$C^1_s(t) = [-2L^*_C, \delta_{sl}(t))$$
and
$$C^2_{s}(t) = [\delta_{sr}(t), 2L^*_C),$$
be the structure of the two communities $C_1(t)$ and $C_2(t)$, $t \geq 0$,  as defined in Section~\ref{sec:model}. The optimal content $\xs_1(y)$ that agent $y \in C^1_s \cup C^2_s$ produces in community $C_1(t)$ is given by
$$ \xs_1(y) = -L^*_C + \frac{\delta_{dl}(t)}{2} = \argmax_{x \in \mathcal{R}} q(x|y) \int_{z \in C^1_d(t)} p(x|z)dx,$$
and the optimal content $\xs_2(y)$ that agent $y\in C^1_s \cup C^2_s$ produces in community $C_2(t)$ is given by
$$ \xs_2(y) =  L^*_C + \frac{\delta_{dr}(t)}{2} = \argmax_{x \in \mathcal{R}} q(x|y) \int_{z \in C^2_d(t)} p(x|z)dx.$$
\end{lemma}

Using the result of Lemma~\ref{lemma:xs_t}, we obtain the following result that characterizes the utility that agents $y \in C^1_d(t)$ receive in community $C_1(t)$.
\begin{lemma}\label{lemma:Ud_t}
Let the community $C_1(t)$, $t \geq 0$, given by
$$C^1_d(t) = [-L^*_C - l^*_d, -L^*_C + l^*_d + \delta_{dl}(t) )$$
and
$$C^1_s(t) = [-2L^*_C, \delta_{sl}(t))$$
as defined in Section~\ref{sec:model}. If  at time $t \geq 0$ we have that
$$ \delta_{sl}(t) > - 2L^*_C,$$
then the following is true.
\begin{enumerate}
\item[a)] If 
$$ 0 < 2l^*_d + \delta_{dl}(t) < 2 \left ( \frac{f_0}{a} - \frac{c}{a g_0} \right ),$$
then we have that
$$ U_{C_1(t)}^{(d)} ( -L^*_C + l^*_d + \delta_{dl}(t) ) > 0.$$
\item[b)] If 
$$ 0 < 2l^*_d + \delta_{dl}(t) = 2 \left ( \frac{f_0}{a} - \frac{c}{a g_0} \right )$$
then we have that
$$ U_{C_1(t)}^{(d)} ( -L^*_C + l^*_d + \delta_{dl}(t) ) = 0.$$
\item[c)] If 
$$ 2l^*_d + \delta_{dl}(t) > 2 \left ( \frac{f_0}{a} - \frac{c}{a g_0} \right ),$$
then we have that
$$ U_{C_1(t)}^{(d)} ( -L^*_C + l^*_d + \delta_{dl}(t) ) < 0.$$
\end{enumerate}
\end{lemma}

Similarly, we obtain the following result that characterizes the utility that agents $y \in C^2_d(t)$ receive in  community $C_2(t)$.
\begin{lemma}\label{lemma:Us_t}
Let  the community $C_w(t)$, $t \geq 0$, given by
$$C^2_d(t) = [L^*_C - l^*_d + \delta_{dr}(t), L^*_C + l^*_d )$$
and
$$C^2_s(t) = [\delta_{sr}(t),2L^*_C)$$
be as defined in Section~\ref{sec:model}. If  at time $t \geq 0$ we have that
$$ \delta_{sr}(t) < 2L^*_C,$$
then the following is true:
\begin{enumerate}
\item[c)] If 
$$ 0 < 2l^*_d - \delta_{dr}(t) < 2 \left ( \frac{f_0}{a} - \frac{c}{a g_0} \right ),$$
then we have that
$$ U_{C_2(t)}^{(d)} ( L^*_C - l^*_d + \delta_{dr}(t) ) > 0.$$
\item[c)] If 
$$ 0 < 2l^*_d - \delta_{dr}(t) = 2 \left ( \frac{f_0}{a} - \frac{c}{a g_0} \right )$$
then we have that
$$ U_{C_2(t)}^{(d)} ( L^*_C - l^*_d + \delta_{dr}(t) ) = 0.$$
\item[c)] If 
$$ 2l^*_d - \delta_{dr}(t) > 2 \left ( \frac{f_0}{a} - \frac{c}{a g_0} \right ),$$
then we have that
$$ U_{C_2(t)}^{(d)} ( L^*_C - l^*_d + \delta_{dr}(t) ) < 0.$$
\end{enumerate}
\end{lemma}
The above results follow immediately from Lemma~\ref{lemma:xs_t} that characterizes the optimal content that agents produce for the case where the function $f$ and $g$ are as given by Assumption~\ref{assumption:simple_fg} and we have 
$$f(x) = \mathrm{max} \{ 0, f_0 - ax \}$$
and
$$g(x) = g_0,$$ where $g_0 \in (0,1]$,
and the definition of the utility functions as given in Section~\ref{sec:model}.

\newpage
\section{Properties of the Perturbation Functions}\label{app:C1_C2}

In this appendix we derive  properties of the perturbation functions $\ddl(t)$, $\ddr(t)$, $\dsl(t)$, $\dsr(t)$ defined by Eq.~\eqref{eqn:dl_deri}~-~Eq.~\eqref{eqn:sr_deri} in Section~\ref{sec:model}.
Our first result shows that the perturbation functions are given by continuous functions.

\begin{lemma}\label{lemma:dl_cont}
Let the communities  $C_1(t)$ and $C_2(t)$, $t \geq 0$,  and  perturbation functions $\ddl(t)$, $\ddr(t)$, $\dsl(t)$, $\dsr(t)$ be as defined in Section~\ref{sec:model}. Then the perturbation functions $\ddl(t)$, $\ddr(t)$, $\dsl(t)$, $\dsr(t)$ are continuous in $t$ for $t \geq 0$.
\end{lemma}  

\begin{proof}
Using Lemma~\ref{lemma:xs_t}, the utility $U_{C_1(t)}^{(d)}(y)$, $y \in \setR$, $t \geq 0$, is given by
$$ U_{C_1(t)}^{(d)}(y) =  E_p E_q (2\LsC + \dsl(t)) \left [ g_0 \max \left \{ 0,f_0- a \Big \lvert \frac{- 2\LsC + \ddl(t) }{2} - y \Big \rvert \right \}   - c \right ].$$
Note that
$$ | 2\LsC + \dsl(t) | \leq 2L,$$
ad it follows that
$$ \Big | U_{C_1(t)}^{(d)}(y) \Big | \leq E_p E_q (2L) \max \{c,g_0 a_0 \}, \qquad y \in \setR.$$
Using the same argument, we have  for $y \in \setR$ and $t \geq 0$ that
$$ \Big | U_{C_1(t)}^{(d)}(y) \Big |, \Big | U_{C_1(t)}^{(s)}(y) \Big |, \Big | U_{C_2(t)}^{(d)}(y) \Big |, \Big | U_{C_2(t)}^{(s)}(y) \Big |    \leq  E_p E_q (2L) \max\{c,g_0 a_0 \}.$$
It then follows from definition of  the perturbation functions $\ddl(t)$, $\ddr(t)$, $\dsl(t)$, $\dsr(t)$ given by Eq.~\eqref{eqn:dl_deri}~-~Eq.~\eqref{eqn:sr_deri} that they are continuous in $t$ for $t \geq 0$. 
\end{proof}

Our next result characterizes the dynamics of the perturbations $\ddl(t)$ and $\ddr(t)$.

\begin{lemma}\label{lemma:dt_U_d}
Let the communities  $C_1(t)$ and $C_2(t)$, $t \geq 0$,  be as defined in Section~\ref{sec:model}. If we have that
$$ U_{C_1(t)}^{(d)} ( -\LsC + \lsd + \ddl(t) ) \geq 0, \qquad t \geq 0,$$
and
$$ U_{C_2(t)}^{(d)} ( -\LsC + \lsd + \ddl(t) ) \geq 0, \qquad t \geq 0,$$
as well as 
$$U_{C_1(t)}^{(d)} ( \LsC - \lsd + \ddr(t) ) \geq 0 , \qquad t \geq 0,$$
and
$$U_{C_2(t)}^{(d)} ( \LsC - \lsd + \ddr(t) ) \geq 0 , \qquad t \geq 0,$$
then we have that
$$  \frac{ d \ddl(t)}{dt} = \frac{ d \ddr(t)}{dt}, \qquad t \geq 0.$$
and
$$\ddl(t) = \ddr(t), \qquad t \geq 0.$$
\end{lemma}

\begin{proof}
  Using the conditions in the statement of the lemma, we have for the perturbation dynamics given by Eq.~\eqref{eqn:dl_deri}~and~Eq.~\eqref{eqn:dr_deri} in Section~\ref{sec:model} that
$$\frac{ d \ddl(t)}{dt}=  U_{C_1(t)}^{(d)} ( -L_C + \lsd + \ddl(t) ) - U_{C_2(t)}^{(d)} ( -L_C + \lsd + \ddl(t)) $$
and
$$\frac{ d \ddr(t)}{dt}=  U_{C_1(t)}^{(d)} ( L_C - \lsd + \ddr(t) ) - U_{C_2(t)}^{(d)} ( L_C - \lsd + \ddr(t) ).$$
It then follows that
$$  \frac{ d \ddl(t)}{dt} = \frac{ d \ddr(t)}{dt}, \qquad t \geq 0.$$
Furthermore, we have  by the definition of the perturbation functions in Section~\ref{sec:model}  that
$$\ddl(0) = \ddr(0).$$
Combining these two results, it follows that
$$\ddl(t) = \ddr(t), \qquad t \geq 0,$$
\end{proof}  

Using the same argument as given in the proof for Lemma~\ref{lemma:dt_U_s}, we obtain the following result that characterizes the dynamics of the perturbations $\dsl(t)$ and $\dsr(t)$.

\begin{lemma}\label{lemma:dt_U_s}
Let the communities  $C_1(t)$ and $C_2(t)$, $t \geq 0$,  be as defined in Section~\ref{sec:model}. If we have that
$$ U_{C_1(t)}^{(s)} ( \dsl(t) ) \geq 0, \qquad t \geq 0,$$
and
$$ U_{C_2(t)}^{(s)} ( \dsl(t) ) \geq 0, \qquad t \geq 0,$$
as well as 
$$U_{C_1(t)}^{(s)} ( \dsr(t) ) \geq 0 , \qquad t \geq 0,$$
and
$$U_{C_2(t)}^{(s)} ( \dsr(t) ) \geq 0 , \qquad t \geq 0,$$
then we have that
$$  \frac{ d \dsl(t)}{dt} = \frac{ d \dsr(t)}{dt}, \qquad t \geq 0.$$
and
$$\dsl(t) = \dsr(t), \qquad t \geq 0.$$
\end{lemma}

The next two results provide a condition under which the absolute values $|\ddl(t)|$ and $|\ddr(t)|$  can not become too large. The first of the two results provides a condition for the perturbation  $\ddl(t)$.

\begin{lemma}\label{lemma:dt_close_1}
Let  the community $C_1(t)$, $t \geq 0$,   be as defined in Section~\ref{sec:model}. If we have that  
$$ 2\lsd + \ddl(0) < 2 \left ( \frac{f_0}{a} - \frac{c}{a g_0} \right ),$$
then we have that
$$   2\lsd + \ddl(t) \leq 2 \left ( \frac{f_0}{a} - \frac{c}{a g_0} \right ), \quad t \geq 0.$$
\end{lemma}

\begin{proof}
  We prove the result by contradiction. To do that, suppose that the result of the lemma is not true. As by Lemma~\ref{lemma:dl_cont} we have that the perturbation function $\ddl(t)$ is continuous in $t$ for $t \geq 0$, and by  assumption we have 
$$ 2\lsd + \ddl(0) < 2 \left ( \frac{f_0}{a} - \frac{c}{a g_0} \right ),$$
it then follows that there exists a  time $t_1 > 0$ such that
$$2\lsd + \ddl(t_1) = 2 \left ( \frac{f_0}{a} - \frac{c}{a g_0} \right ),$$
and
\begin{equation}\label{eq:t_1}
\frac{ d \ddl(t=t_1)}{dt} > 0.
\end{equation}. 
Using Lemma~\ref{lemma:Ud_t}, we have for the time $t_1$  that
$$ U_{C_1(t_1)}^{(d)} ( -\LsC + \lsd + \ddl(t_1) ) = 0.$$
Using this result in the definition of the perturbation functions given by Eq.~\eqref{eqn:dl_deri} we obtain that
$$\frac{ d \ddl(t=t_1)}{dt}=  - \max \left \{ 0, U_{C_2(t_1)}^{(d)} ( -L_C + \lsd + \ddl(t_1)) \right \} \leq 0. $$
This leads to a contraction with Eq.~\eqref{eq:t_1} which states that
$$\frac{ d \ddl(t=t_1)}{dt} > 0.$$
It then follows that
$$   2\lsd + \ddl(t) \leq 2 \left ( \frac{f_0}{a} - \frac{c}{a g_0} \right ), \quad t \geq 0,$$
and we obtain the result of the lemma.
\end{proof}

Using the same argument as given to prove Lemma~\ref{lemma:dt_close_1}, we obtain the following condition to guarantee that the perturbation $\ddr(t)$ can not become too large in magnitude.
\begin{lemma}\label{lemma:dt_close_2}
Let the community  $C_2(t)$, $t \geq 0$,  be  as defined in Section~\ref{sec:model}. If we have that

$$  2\lsd - \ddr(0) < 2 \left ( \frac{f_0}{a} - \frac{c}{a g_0} \right ), $$
then we have that
$$  2\lsd - \ddr(t) \leq 2 \left ( \frac{f_0}{a} - \frac{c}{a g_0} \right ), \quad t \geq 0.$$
\end{lemma}

Using Lemma~\ref{lemma:dt_close_1} and  Lemma~\ref{lemma:dt_close_2}, the next result provides a condition for the utilities obtained by the border agents to always be non-negative.

\begin{lemma}\label{lemma:Ut_no_gap}
  Let the community   $C_1(t)$ and $C_2(t)$, $t \geq 0$,  be  as defined in Section~\ref{sec:model}. If we have that
  $$  2\lsd + \ddl(0) < 2 \left ( \frac{f_0}{a} - \frac{c}{a g_0} \right )$$
and
$$  2\lsd - \ddr(0) < 2 \left ( \frac{f_0}{a} - \frac{c}{a g_0} \right ), $$
then we have that
$$ U_{C_1(t)}^{(d)} ( -\LsC + \lsd + \ddl(t) ) \geq 0, \qquad t \geq 0,$$
and
$$ U_{C_2(t)}^{(d)} ( \LsC - \lsd + \ddr(t) ) \geq 0, \qquad t \geq 0,$$
as well as 
$$U_{C_1(t)}^{(s)} ( -\LsC + \lsd + \ddl(t) ) \geq 0 , \qquad t \geq 0,$$
and
$$U_{C_2(t)}^{(s)} ( \LsC - \lsd + \ddr(t) ) \geq 0 , \qquad t \geq 0.$$
\end{lemma}

\begin{proof}
  Using the result of Lemma~\ref{lemma:dt_close_1} and  Lemma~\ref{lemma:dt_close_2}, we obtain under the conditions given in the statement of the lemma
  $$  2\lsd + \ddl(t) \leq 2 \left ( \frac{f_0}{a} - \frac{c}{a g_0} \right ), \quad t \geq 0$$
  and
  $$  2\lsd - \ddr(t) \leq 2 \left ( \frac{f_0}{a} - \frac{c}{a g_0} \right ), \quad t \geq 0.$$
  Using the result of Lemma~\ref{lemma:Ud_t} and  Lemma~\ref{lemma:Us_t}, we obtain in this case for all agents $y \in C^1_d(t)$ that
  $$ U_{C_1(t)}^{(d)} (y) \geq 0, \quad t \geq 0,$$
  and for all agents  $y \in C^2_d(t)$ that
  $$ U_{C_2(t)}^{(d)} (y ) \geq 0, \quad t \geq 0.$$
  Using this result in the definition of the utility functions $ U_{C_1(t)}^{(s)}(y)$ and $ U_{C_1(t)}^{(s)}(y)$, we obtain for all agents $y \in C^1_s(t)$ that
  $$ U_{C_1(t)}^{(s)} (y) \geq 0, \quad t \geq 0,$$
  and for all agents  $y \in C^2_d(t)$ that
  $$ U_{C_2(t)}^{(s)} (y ) \geq 0, \quad t \geq 0.$$
  In particular, these results imply that

$$ U_{C_1(t)}^{(d)} ( -\LsC + \lsd + \ddl(t) ) \geq 0, \qquad t \geq 0,$$
and
$$ U_{C_2(t)}^{(d)} ( \LsC - \lsd + \ddr(t) ) \geq 0, \qquad t \geq 0,$$
as well as 
$$U_{C_1(t)}^{(s)} ( -\LsC + \lsd + \ddl(t) ) \geq 0 , \qquad t \geq 0,$$
and
$$U_{C_2(t)}^{(s)} ( \LsC - \lsd + \ddr(t) ) \geq 0 , \qquad t \geq 0.$$

The result of the lemma then follows.
\end{proof}

\begin{lemma}\label{lemma:dt_outside_1}
  Let $\mathcal{C}(\LsC,\lsd)$ be a Nash equilibrium as given by Proposition~\ref{prop:NE_ld}, and let the corresponding community $C_1(t)$, $t \geq 0$, be as defined in Section~\ref{sec:model}. If we have that
$$ 2 \lsd <  2\lsd + \ddl(0) < 2 \LsC$$
and
$$ 2 \lsd <  2\lsd - \ddr(0) < 2 \LsC,$$
then we have that
$$  2 \lsd \leq  2\lsd + \ddl(t) \leq  2\lsd + \ddl(0), \quad t \geq 0,$$
and
$$  2 \lsd \leq  2\lsd - \ddr(t) \leq  2\lsd + \ddr(0), \quad t \geq 0.$$
\end{lemma}

\begin{proof}
Using Lemma~\ref{lemma:Ud_t}, if for time $t \geq 0$ we have that
$$ 2\lsd + \ddl(t) \geq 2 \lsd = 2 \left ( \frac{f_0}{a} - \frac{c}{a g_0} \right ), $$
then we obtain that 
$$ U_{C_1(t)}^{(d)} ( -\LsC + \lsd + \ddl(t) ) \leq 0.$$
Using this result in the definition of the perturbation function $\ddl(t)$  given by Eq.~\eqref{eqn:dl_deri},  if for time $t \geq 0$ we have that
$$ 2\lsd + \ddl(t) \geq 2 \lsd, $$
then we obtain that 
$$\frac{ d \ddl(t)}{dt}=  U_{C_1(t)}^{(d)} ( -\LsC + \lsd + \ddl(t))  - \max \left \{ 0, U_{C_2(t)}^{(d)} ( -L_C + \lsd + \ddl(t)) \right \} \leq 0. $$
This implies that
\begin{equation}\label{eq:LsC1}
2\lsd + \ddl(t) \leq  2\lsd + \ddl(0) < 2 \LsC, \quad t \geq 0.
\end{equation}  
Using the same argument we have that if
$$ 2 \lsd <  2\lsd - \ddr(0) < 2 \LsC,$$
then we obtain that 
\begin{equation}\label{eq:LsC2}
2\lsd - \ddr(t) \leq  2\lsd + \ddr(0) < 2 \LsC, \quad t \geq 0.
\end{equation}  
Therefore, it remains to show that
$$  2 \lsd \leq  2\lsd + \ddl(t), \quad t \geq 0,$$
and
$$  2 \lsd \leq  2\lsd - \ddr(t), \quad t \geq 0.$$
Using Eq.~\eqref{eq:LsC1}~and~\eqref{eq:LsC2}, as well as Lemma~\ref{lemma:xs_t}, in the definition of the utility function $U_{C_2(t)}^{(d)} (y)$, we obtain that 
$$ U_{C_2(t)}^{(d)} ( -L_C + \lsd + \ddl(t)) < 0, \qquad t \geq 0.$$
Using Eq.~\eqref{eqn:dl_deri} in Section~\ref{sec:model}, it then follows that
$$\frac{ d \ddl(t)}{dt}=  U_{C_1(t)}^{(d)} ( -\LsC + \lsd + \ddl(t)), \qquad t \geq 0. $$
Moreover by Lemma~\ref{lemma:Ud_t}, if for time $t \geq 0$ we have that
$$ 2\lsd + \ddl(t) \leq 2 \lsd, $$
then we obtain that 
$$ U_{C_1(t)}^{(d)} ( -\LsC + \lsd + \ddl(t) ) \geq 0.$$
Combining the two results above, we obtain that if
$$ 2 \lsd <  2\lsd + \ddl(0) < 2 \LsC,$$
then we have that
$$  2 \lsd \leq  2\lsd + \ddl(t), \quad t \geq 0.$$
Using the same argument we can show that if
$$ 2 \lsd <  2\lsd - \ddl(0) < 2 \LsC,$$
then we obtain that
$$  2 \lsd \leq  2\lsd - \ddr(t), \quad t \geq 0.$$
The result of the lemma then follows. 
\end{proof}

The next result characterizes the perturbation functions for the case where the  Nash equilibrium $\mathcal{C}(\LsC,\LsC)$ is as given by Proposition~\ref{prop:NE_LC}.
\begin{lemma}\label{lemma:dt_no_gap}
  Let $\mathcal{C}(\LsC,\LsC)$ be a Nash equilibrium as given by Proposition~\ref{prop:NE_LC} with
$$\LsC < \frac{f_0}{a} - \frac{c}{a g_0},$$
and let  the corresponding communities  $C_1(t)$ and  $C_2(t)$, $t \geq 0$, be as defined in Section~\ref{sec:model}. Furthermore, let
$$K = E_p E_q \Bsbl -2c +2f_0 g_0 - 2\LsC a g_0 \Bsbr$$
and
$$M= E_p E_q \Bsbl  2\LsC a g_0 \Bsbr .$$
Then the following is true.
If 
$$ 2\lsd + \ddl(0) < 2 \left ( \frac{f_0}{a} - \frac{c}{a g_0} \right )$$
and
$$  2\lsd - \ddr(0) < 2 \left ( \frac{f_0}{a} - \frac{c}{a g_0} \right ), $$
then we have that
$$\ddl(t) = \ddr(t) =  \delta_{d}(t), \qquad t \geq 0,$$
and
$$\dsl(t) = \dsr(t) =  \delta_{s}(t), \qquad t \geq 0,$$
as well as
$$\frac{\ddl(t)}{\delta t} = \frac{\ddr(t)}{\delta t} =  \frac{\delta_{d}(t)}{\delta t} = K \delta_s(t)  - M \delta_d(t) , \qquad t \geq 0,$$
and
$$\frac{\dsl(t)}{\delta t} = \frac{\dsr(t)}{\delta t} =  \frac{\delta_{s}(t)}{\delta t} =  K \delta_d(t), \qquad t \geq 0.$$
\end{lemma}

\begin{proof}
  Using the result of Lemma~\ref{lemma:Ut_no_gap}, we obtain under the conditions given in the statement of the lemma that

  $$ U_{C_1(t)}^{(d)} ( -\LsC + \lsd + \ddl(t) ) \geq 0, \qquad t \geq 0,$$
  and
  $$ U_{C_2(t)}^{(d)} ( \LsC - \lsd + \ddr(t) ) \geq 0, \qquad t \geq 0,$$
  as well as 
  $$U_{C_1(t)}^{(s)} ( -\LsC + \lsd + \ddl(t) ) \geq 0 , \qquad t \geq 0,$$
  and
  $$U_{C_2(t)}^{(s)} ( \LsC - \lsd + \ddr(t) ) \geq 0 , \qquad t \geq 0.$$

Combining this result with Lemma~\ref{lemma:dt_U_d} we obtain that
$$  \frac{ d \ddl(t)}{dt} = \frac{ d \ddr(t)}{dt} = \frac{ d \delta_{d}(t)}{dt} , \qquad t \geq 0.$$
and
$$\ddl(t) = \ddr(t) = \delta_{d}(t), \qquad t \geq 0.$$
Similarly, using the of  Lemma~\ref{lemma:dt_U_d} we obtain that
$$  \frac{ d \dsl(t)}{dt} = \frac{ d \dsr(t)}{dt} = \frac{ d \delta_{s}(t)}{dt} , \qquad t \geq 0.$$
and
$$\dsl(t) = \dsr(t) = \delta_{s}(t), \qquad t \geq 0.$$

Finally, using the result of  Lemma~\ref{lemma:xs_t} which characterizes the content that agents in the sets  $C^1_s(t)$ and  $C^2_s(t)$, we obtain that
$$ U_{C_1(t)}^{(d)} (- \LsC + \lsd + \delta_{d}(t)) 
=
E_p E_q (2\LsC + \delta_s(t)) \left [ g_0 \left ( f_0- a \left [ \LsC+\frac{\delta_d(t)}{2} \right ] \right ) - c \right ] \geq 0, \qquad t \geq 0,
$$
and
$$
U_{C_2(t)}^{(d)} ( \LsC - \lsd + \delta_{d}(t)) 
=
E_p E_q (2\LsC - \delta_s(t)) \left [ g_0 \left ( f_0- a \left [ \LsC - \frac{\delta_d(t)}{2} \right ] \right ) - c \right ] \geq 0, \qquad t \geq 0.
$$
Using these equation in the dynamics of the perturbation $\ddl(t)$  and
$\dsl(t)$ given by Eq.~\eqref{eqn:dl_deri}~-~Eq.~\eqref{eqn:sr_deri}, we obtain that
$$\frac{ d \delta_{d}(t)}{dt} = E_p E_q \left [ \Bbl  -2c +2f_0 g_0 - 2\LsC a g_0 \Bbr \delta_s(t)  -  2\LsC a g_0 \delta_d(t) \right ], \qquad t \geq 0.$$
Setting
$$K = E_p E_q \Bsbl -2c +2f_0 g_0 - 2\LsC a g_0 \Bsbr$$ and
$$M= E_p E_q \Bsbl 2\LsC a g_0 \Bsbr,$$
we then obtain that
\begin{equation*}
\frac{ d \delta_{d}(t)}{dt} =
 \Bsbl K \delta_s(t)  - M \delta_d(t) \Bsbr, \qquad t \geq 0.
\end{equation*}
Similarly, we obtain that
$$
U_{C^1}^{(s)} (\delta_s(t))
=
E_p E_q (2\LsC + \delta_d(t)) \left [  g_0 \left ( \frac{4 f_0 -  a \left [2\LsC + \delta_d(t)  \right ]}{4} - c \right )   \right ] \geq 0, \qquad t \geq 0,
$$
and
$$
U_{C^2}^{(s)} (\delta_s(t))
=
E_p E_q (2\LsC - \delta_d(t)) \left [  g_0  \left ( \frac{4 f_0 -  a \left [2\LsC - \delta_d(t)  \right ]}{4} - c \right )   \right ] \geq 0, \qquad t \geq 0,
$$
as well as
$$\frac{ d \delta_{s}(t)}{dt} =  E_p E_q   \Bbl  -2c +2f_0 g_0 - 2\LsC a g_0 \Bbr  \delta_d(t), \qquad t \geq 0.$$
Using the definitions for $K$ as given above, we can re-write this equation as
\begin{equation*}
\frac{ d \delta_{s}(t)}{dt} = K \delta_d(t), \qquad t \geq 0.
\end{equation*}
The result of the lemma then follows.
\end{proof}

The next result characterizes the perturbation functions $\dsl(t)$ and $\dsr(t)$ for the case where the  Nash equilibrium $\mathcal{C}(\LsC,\lsd)$ is as given by Proposition~\ref{prop:NE_ld}.

\begin{lemma}\label{lemma:dt_s_gap}
  Let $\mathcal{C}(\LsC,\lsd)$ be a Nash equilibrium as given by Proposition~\ref{prop:NE_ld}, and let  the corresponding communities  $C_1(t)$ and  $C_2(t)$, $t \geq 0$, be as defined in Section~\ref{sec:model}. If we have that  
$$U^{(s)}_{C_1(t)}(\dsl(t)) > 0, \qquad t \geq 0,$$
and
$$U^{(s)}_{C_2(t)}(\dsr(t)) > 0, \qquad t \geq 0,$$
as well as
$$ - 2 \lsd < \ddl(t) < 2 \left [ \frac{f_0}{a} - \lsd \right ], \qquad t \geq 0,$$
and
$$2 \lsd > \ddr(t) > - 2  \left [ \frac{f_0}{a} - \lsd \right ], \qquad t \geq 0,$$
then we have that
$$ \left | \frac{\dsl(t)}{\delta t} \right |  =  \left | \frac{\dsr(t)}{\delta t} \right |
 \leq
  E_p E_q a g_0 \frac{1}{4}  \Bsbl \ddl^2(t) + \ddr^2(t) \Bsbr.
$$
\end{lemma}

\begin{proof}
Note that if for a given time $t$, $t \geq 0$, we have
$$ - 2\lsd < \ddl(t) \leq 0,$$
then we have from Lemma~\ref{lemma:xs_t} that
\begin{eqnarray}\label{eq:U_1_1}
  U^{(s)}_{C_1(t)}(\dsl(t)) = U^{(s)}_{C_1(t)}(\dsr(t)) 
  &=& E_p E_q \Bsbl g_0f_0   - c \Bsbr \lsd \\ \nonumber
&& - a g_0 \left ( \frac{\ddl(t)}{2} \right )^2.
\end{eqnarray}
Similarly, if for a given time $t$, $t \geq 0$ we have
  $$ 0 <  \ddl(t)  < 2 \left [ \frac{f_0}{a} - \lsd \right ], \qquad t \geq 0, $$
then we have that
\begin{eqnarray}
  U^{(s)}_{C_1(t)}(\dsl(t)) = U^{(s)}_{C_1(t)}(\dsr(t)) 
  &=& E_p E_q \Bsbl g_0f_0   - c \Bsbr \lsd \\ \nonumber
&& - a g_0 \left ( \frac{\ddl(t)}{2} \right )^2.
\end{eqnarray}

Using the same argument,  if for a given time $t$, $t \geq 0$ we have
$$ - 2 \left [ \frac{f_0}{a} - \lsd \right ] < \ddr(t) < 2\lsd,$$
then we have that
\begin{eqnarray}\label{eq:U_2_1}
  U^{(s)}_{C_2(t)}(\dsr(t)) = U^{(s)}_{C_2(t)}(\dsl(t)) 
  &=& E_p E_q \Bsbl g_0f_0 - c \Bsbr \lsd \\ \nonumber
&& - a g_0 \left ( \frac{\ddr(t)}{2} \right )^2.
\end{eqnarray}

As by assumption we have that
  $$U^{(s)}_{C_1(t)}(\dsl(t)) = U^{(s)}_{C_1(t)}(\dsr(t)) > 0,$$
  and
  $$U^{(s)}_{C_2t)}(\dsl(t)) = U^{(s)}_{C_2(t)}(\dsr(t))  > 0,$$
it follows that
$$ \frac{d\dsl(t)}{dt} = U^{(s)}_{C_1(t)}(\dsl(t)) - U^{(s)}_{C_2(t)}(\dsl(t)) =  -  \frac{d\dsr(t)}{dt}.$$
 
Combining this result with Eq.~\eqref{eq:U_1_1}~-~\eqref{eq:U_2_1}, we obtain that
$$ \left | U^{(s)}_{C_1(t)}(\dsl(t)) - U^{(s)}_{C_2(t)}(\dsl(t)) \right | \leq
 E_p E_q a g_0 \frac{1}{4}  \Bsbl \ddl^2(t) + \ddr^2(t) \Bsbr,
$$
and
$$ \left | \frac{\dsl(t)}{\delta t} \right |  =  \left | \frac{\dsr(t)}{\delta t} \right |
\leq
 E_p E_q a g_0 \frac{1}{4}  \Bsbl \ddl^2(t) + \ddr^2(t) \Bsbr.
$$
\end{proof}



\newpage
\section{Sufficient Condition for a Neutral-Stable Equilibrium}

The next result provides a condition for Nash equilibrium  $\mathcal{C}(\LsC,\lsd)$ as given by Proposition~\ref{prop:NE_ld} to be neutral-stable.

\begin{lemma}\label{lemma:neutral}
  Let $\mathcal{C}(\LsC,\lsd)$ be a Nash equilibrium as given by Proposition~\ref{prop:NE_ld}. Furthermore let the corresponding communities $C_1(t)$ and  $C_2(t)$, $t \geq 0$, be as defined in Section~\ref{sec:model}. Then the following is true.
If 
$$ \lim_{t \to \infty}  \ddl(t) =  \lim_{t \to \infty}  \ddr(t) = 0,$$
then we have that
$$ \lim_{t \to \infty} U_{C_1(t)}^{(s)}(y) =  \lim_{t \to \infty}  U_{C_2(t)}^{(s)}(y) > 0, \qquad y \in C^1_s \cup  C^2_s.$$  
\end{lemma}  

\begin{proof}
  By Lemma~\ref{lemma:xs_t} we have that the  optimal content to produce in community $C_1(t)$ is the same for all agents $y \in C^1_s \cup C^2_s$, and given by
  $$ \xs_1(y,t) = -\LsC + \frac{\ddl(t)}{2}.$$
Similarly,  the  optimal content to produce in community $C_2(t)$ is the same for all agents $y \in C^1_s \cup C^2_s$, and given by
$$  \xs_2(y,t) = \LsC - \frac{\ddl(t)}{2}.$$
If we have that
$$ \lim_{t \to \infty}  \ddl(t) =  \lim_{t \to \infty}  \ddr(t) = 0,$$
then it follows
$$ \lim_{t \to \infty} \xs_1(y,t) = -\LsC$$
and
$$ \lim_{t \to \infty} \xs_2(y,t) = \LsC.$$
As by Assumption~\ref{assumption:simple_fg} we have that
$$ g(x) = g_0, \qquad x \geq 0,$$
it follows that for all agents $y \in C^1_s \cup C^2_s$ we have that
\begin{eqnarray*}
  \lim_{t \to \infty} U_{C_1(t)}^{(s)}(y) &=&   =  E_p E_q \int_{z \in C^1_d} \Bsbl g_0 f(|-\LsC - z|) - c \Bsbr dz \\
  &=&   E_p E_q \int_{-\LsC - \lsd}^{-\LsC + \lsd} \Bsbl g_0 f(|-\LsC - z|) - c \Bsbr dz \\
  &=&   E_p E_q \int_{-\LsC - \lsd}^{-\LsC + \lsd} \Bsbl g_0 f(|\LsC + z|) - c \Bsbr dz \\
  &=&   E_p E_q \int_{\LsC - \lsd}^{\LsC + \lsd} \Bsbl g_0 f(|\LsC - z|) - c \Bsbr dz \\
  &=&   \lim_{t \to \infty} U_{C_2(t)}^{(s)}(y).
\end{eqnarray*}
Using this result, we obtain that
$$ \lim_{t \to \infty} U_{C_2(t)}^{(s)}(y) =  \lim_{t \to \infty} U_{C_2(t)}^{(s)}(y).$$
As the Nash equilibrium  $\mathcal{C}(\LsC,\lsd)$  is as given by Proposition~\ref{prop:NE_ld}, we have that
$$\lsd  = \frac{f_0}{a} - \frac{c}{a g_0} < \LsC.$$
Furthermore, we have by Assumption~\ref{assumption:simple_fg} that
$$f(x) = \mathrm{max} \{ 0, f_0 - ax \},$$
where $f_0 \in (0,1]$ and  $a >0$.
Combining these two results, we obtain that
$$f(0) = f_0$$
and
$$f(\lsd) = \frac{c}{g_0}.$$
It then follow that
$$   \lim_{t \to \infty} U_{C_1(t)}^{(s)}(y) = E_p E_q \int_{-\LsC - \lsd}^{-\LsC + \lsd}
\Bsbl g_0 f(|-\LsC - z|) - c \Bsbr dz = E_p E_q \lsd \Bsbl f_0g_0 - c \Bsbr.$$
As by Assumption~\ref{assumption:simple_fg} we have that
$$  f_0g_0 - c > 0,$$
we obtain that
$$ \lim_{t \to \infty} U_{C_2(t)}^{(s)}(y) =  \lim_{t \to \infty} U_{C_2(t)}^{(s)}(y) > 0.$$
The result of the lemma then follows. 
\end{proof}

Next we provide two conditions under which we have for Nash equilibrium  $\mathcal{C}(\LsC,\lsd)$ as given by Proposition~\ref{prop:NE_ld} that
$$ \lim_{t \to \infty}  \ddl(t) = 0.$$
The first result considers the case where
$$\ddl(0) < 0.$$

\begin{lemma}\label{lemma:ddl_0_1}
 Let $\mathcal{C}(\LsC,\lsd)$ be a Nash equilibrium as given by Proposition~\ref{prop:NE_ld}, and let  the corresponding communities $C_1(t)$ and  $C_2(t)$, $t \geq0$, be as defined in Section~\ref{sec:model}. Furthermore, let
 $$\delta = \LsC - \lsd$$
 and
 $$B_0 = \frac{E_p E_q \LsC a g_0}{2}.$$
Then the following is true. 
If 
$$ - \dl < \ddl(0) < 0 $$
and
$$  - \dl < \ddr(0),$$
as well as
$$  |\dsl(t)| < \frac{\LsC}{2}, \qquad t \geq 0,$$
then we have that
$$ -\dl  B_0 e^{-B_0t} \leq \ddl(t) \leq 0, \qquad t \geq 0,$$
and
$$\lim_{t \to \infty}  \ddl(t) = 0.$$
\end{lemma}

\begin{proof}
  Note that  for a Nash equilibrium  $\mathcal{C}(\LsC,\lsd)$ as given by Proposition~\ref{prop:NE_ld} we have that
$$ 0 < \lsd < \LsC$$
and
$$ \lsd =  \left ( \frac{f_0}{a} - \frac{c}{a g_0} \right ).$$
It then follows that
$$\delta  = \LsC - \lsd > 0.$$
Furthermore, if we have that
$$ \ddl(0) < 0,$$
then we obtain that
$$ 2 \lsd + \ddl(0) < 2 \lsd = 2 \left ( \frac{f_0}{a} - \frac{c}{a g_0} \right ).$$

From Lemma~\ref{lemma:Ut_no_gap} we then have that
\begin{equation}\label{eq:ddl_0_1_small_1}
  2 \lsd - \ddl(t) \leq  2 \lsd = 2 \left ( \frac{f_0}{a} - \frac{c}{a g_0} \right ), \qquad t \geq 0,
\end{equation}  
and it follows that
\begin{equation}\label{eq:ddl_0_1_small_2}
\ddl(t) \leq 0, \qquad t \geq 0.
\end{equation}  
Therefore in order to prove the result of the lemma, it remains to show that under the conditions given in the lemma  we have that
$$ -\dl  B_0 e^{-B_0t} \leq \ddl(t), \qquad t \geq 0.$$
Using the assumption that
$$ - \dl < \ddr(0),$$
we have from Lemma~\ref{lemma:dt_outside_1} that
$$-\dl = - (\LsC - \lsd) < \ddr(0) \leq  \ddr(t), \qquad t \geq 0.$$
It then follows that
\begin{equation}\label{eq:ddl_0_1_ddr}
\LsC - \lsd + \ddr(t) > 0, \qquad t \geq 0.
\end{equation}
Furthermore as by Eq.~\eqref{eq:ddl_0_1_small_2} we have that
$$ \ddl(t) \leq 0, \qquad t \geq 0,$$
it follows that
\begin{equation}\label{eq:ddl_0_1_ddl}
- \LsC + \lsd + \ddl(t) \leq - \LsC + \lsd.
\end{equation}  
Combining Eq.~\eqref{eq:ddl_0_1_ddl}~and~\eqref{eq:ddl_0_1_ddr} with Lemma\ref{lemma:xs_t}~and~\ref{lemma:Us_t}, we obtain that 
$$ U_{C_2(t)}^{(d)} \Bbl -\LsC + \lsd + \ddl(t) \Bbr < 0.$$
Combining this result with Eq.~\eqref{eq:ddl_0_1_small_2}, we obtain that
$$ \frac{\ddl(t)}{\delta t} =  U_{C_1(t)}^{(d)}\Bbl -\LsC + \lsd + \ddl(t) \Bbr = - \frac{1}{2} E_p E_q \Bsbl 2\LsC + \dsl(t) \Bsbr a g_0 \ddl(t), \qquad t \geq 0.$$
As from Eq.~\eqref{eq:ddl_0_1_small_2} we have that
$$\ddl(t) \leq 0, \qquad t \geq 0,$$
and by assumption we have that
$$  |\dsl(t)| < \LsC, \qquad t \geq 0,$$
it follows that
$$ \frac{\ddl(t)}{\delta t} \geq \frac{- E_p E_q \LsC a g_0}{2}  \ddl(t) > 0, \qquad t \geq 0.$$
Recall that $B_0$ is given by
$$B_0 = \frac{E_p E_q \LsC a g_0}{2},$$
and we have
$$ \frac{\ddl(t)}{\delta t} \geq -B_0 \ddl(t), \qquad t \geq 0.$$
The solution to this differential equation with the initial condition
$$\ddl(0) < 0$$
is given by
$$\ddl(t) \geq \ddl(0)  e^{- B_0 t}, \qquad t \geq 0.$$
As by assumption we have that
$$  - \delta < \ddl(0) < 0,$$
it follows that 
$$- \delta  e^{- B_0 t} \leq \ddl(t), \qquad t \geq 0.$$
As by Eq.~\eqref{eq:ddl_0_1_small_2} we have that
$$\ddl(t) \leq 0, \qquad t \geq 0,$$
and
$$\lim_{t \to \infty}  - \delta  e^{- B_0 t} =0,$$
the result of the lemma then follows.
\end{proof}

The next result provides a condition under which we have for Nash equilibrium  $\mathcal{C}(\LsC,\lsd)$ as given by Proposition~\ref{prop:NE_ld} that
$$ \lim_{t \to \infty}  \ddl(t) = 0,$$
for the case where
$$\ddl(0) > 0.$$

\begin{lemma}\label{lemma:ddl_0_2}
 Let $\mathcal{C}(\LsC,\lsd)$ be a Nash equilibrium as given by Proposition~\ref{prop:NE_ld}, and let  the corresponding communities $C_1(t)$ and  $C_2(t)$, $t \geq0$, be as defined in Section~\ref{sec:model}. Furthermore, let
$$\delta = \LsC - \lsd.$$  
Then the following is true. 
If 
$$ 0 < \ddl(0) < \min \left \{ \delta, 2 \left [ \frac{f_0}{a} - \lsd \right ] \right  \},$$
and
$$ 0 < \ddr(0),$$
as well as 
$$  |\dsl(t)| < \frac{\LsC}{2}, \qquad t \geq 0,$$
then we have that
$$ \ddl(t) \leq \dl B_0 e^{-B_0t}, \qquad t \geq 0,$$
and
$$\lim_{t \to \infty}  \ddl(t) = 0.$$
\end{lemma}

\begin{proof}
Note that  for a Nash equilibrium  $\mathcal{C}(\LsC,\lsd)$ as given by Proposition~\ref{prop:NE_ld} we have that
$$ 0 < \lsd < \LsC$$
and
$$ \lsd =  \left ( \frac{f_0}{a} - \frac{c}{a g_0} \right ).$$
It then follows that
$$\delta  = \LsC - \lsd > 0.$$
Furthermore, if we have that
$$ 0< \ddl(0) < \delta,$$
then we obtain that
$$ 2 \lsd < \lsd + \ddl(0) < 2 \LsC.$$
From Lemma~\ref{lemma:dt_outside_1}, we then have that
\begin{equation}\label{eq:ddl_0_2_small_1}
  2 \lsd \leq   2 \lsd - \ddl(t) \leq  2 \lsd + \ddl(0), \qquad t \geq 0,
\end{equation}  
and it follows that
\begin{equation}\label{eq:ddl_0_2_small_2}
  \ddl(t) \geq 0, \qquad t \geq 0.
\end{equation}  
Therefore in order to prove the result of the lemma, it remains to show that under the conditions given in the lemma  we have that
$$ \ddl(t) < \dl  B_0 e^{-B_0t}, \qquad t \geq 0.$$
Using the assumption that
$$ - \dl < \ddr(0),$$
we have from Lemma~\ref{lemma:dt_outside_1} that 
$$-\dl = - (\LsC - \lsd) <  \ddr(t), \qquad t \geq 0.$$
It then follows that
\begin{equation}\label{eq:ddl_0_2_ddr}
\LsC - \lsd + \ddr(t) > 0, \qquad t \geq 0.
\end{equation}
Furthermore as by Eq.~\eqref{eq:ddl_0_2_small_2}  we have that
$$ \ddl(t) < \dl = \LsC - \lsd, \qquad t \geq 0,$$
it follows that
\begin{equation}\label{eq:ddl_0_2_ddl}
- \LsC + \lsd +  \ddl(t) \leq 0.
\end{equation}  
Combining Eq.~\eqref{eq:ddl_0_2_ddl}~and~\eqref{eq:ddl_0_2_ddr} with Lemma\ref{lemma:xs_t}~and~\ref{lemma:Us_t}, we obtain that 
$$ U_{C_2(t)}^{(d)} \Bbl -\LsC + \lsd + \ddl(t) \Bbr < 0.$$
Combining this result with Eq.~\eqref{eq:ddl_0_1_small_2}, we obtain that
$$ \frac{\ddl(t)}{\delta t} =  U_{C_1(t)}^{(d)}\Bbl -\LsC + \lsd + \ddl(t) \Bbr.$$

As by assumption we have that
$$ \ddl(0) < 2 \Bsbl \frac{f_0}{a} - \lsd \Bsbr,$$
and hence by Lemma~\ref{lemma:dt_outside_1} we have
$$ \ddl(t) < 2 \Bsbl \frac{f_0}{a} - \lsd \Bsbr, \qquad t \geq 0,$$
it follows that 
$$ \frac{\ddl(t)}{\delta t} =  U_{C_1(t)}^{(d)}\Bbl -\LsC + \lsd + \ddl(t) \Bbr = - \frac{1}{2} E_p E_q \Bsbl 2\LsC + \dsl(t) \Bsbr a g_0 \ddl(t), \qquad t \geq 0.$$
As we have from Eq.~\eqref{eq:ddl_0_2_small_2} we have that
$$ \ddl(t) \geq 0, \qquad t \geq 0,$$
and by assumption we have that
$$  |\dsl(t)| < \LsC, \qquad t \geq 0,$$
it follows that
$$ \frac{\ddl(t)}{\delta t} \leq \frac{- E_p E_q \LsC a g_0}{2}  \ddl(t) < 0, \qquad t \geq 0.$$
Recall that $B_0$ is given by
$$B_0 = \frac{E_p E_q \LsC a g_0}{2},$$
and we have
$$ \frac{\ddl(t)}{\delta t} \leq - B_0 \ddl(t), \qquad t \geq 0.$$
The solution to this differential equation with the initial condition
$$\ddl(0) > 0$$
is given by
$$\ddl(t) \leq \ddl(0)  e^{- B_0 t}, \qquad t \geq 0.$$
As by assumption we have that
$$ 0 < \ddl(0) < \dl,$$
it follows that 
$$  \ddl(t) <  \delta  e^{- B_0 t}, \qquad t \geq 0.$$
As by Eq.~\eqref{eq:ddl_0_2_small_2} we have that
$$\ddl(t) \geq 0, \qquad t \geq 0,$$
and
$$\lim_{t \to \infty}  - \delta  e^{- B_0 t} =0,$$
the result of the lemma then follows.
\end{proof}

Using the same argument as given in the proof of Lemma~\ref{lemma:ddl_0_1} and Lemma~\ref{lemma:ddl_0_2}, we obtain the following two results.

\begin{lemma}\label{lemma:ddr_0_1}
 Let $\mathcal{C}(\LsC,\lsd)$ be a Nash equilibrium as given by Proposition~\ref{prop:NE_ld}, and let  the corresponding communities $C_1(t)$ and  $C_2(t)$, $t \geq0$, be as defined in Section~\ref{sec:model}. Furthermore, let
 $$\delta = \LsC - \lsd$$
 and
 $$B_0 = \frac{E_p E_q \LsC a g_0}{2}.$$
Then the following is true. 
If 
$$ - \dl < \ddr(0) < 0 $$
and
$$  0 < \ddl(0) < \dl,$$
as well as
$$  |\dsl(t)| < \frac{\LsC}{2}, \qquad t \geq 0,$$
then we have that
$$ -\dl  B_0 e^{-B_0t} \leq \ddr(t) \leq 0, \qquad t \geq 0,$$
and
$$\lim_{t \to \infty}  \ddr(t) = 0.$$
\end{lemma}

\begin{lemma}\label{lemma:ddr_0_2}
 Let $\mathcal{C}(\LsC,\lsd)$ be a Nash equilibrium as given by Proposition~\ref{prop:NE_ld}, and let  the corresponding communities $C_1(t)$ and  $C_2(t)$, $t \geq0$, be as defined in Section~\ref{sec:model}. Furthermore, let
 $$\delta = \LsC - \lsd$$
 and
 $$B_0 = \frac{E_p E_q \LsC a g_0}{2}.$$
Then the following is true. 
If 
$$ 0 < \ddr(0) < \dl $$
and
$$  0 < \ddl(0) < \dl,$$
as well as
$$  |\dsl(t)| < \frac{\LsC}{2}, \qquad t \geq 0,$$
then we have that
$$ \dl  B_0 e^{-B_0t} \geq \ddr(t) \geq 0, \qquad t \geq 0,$$
and
$$\lim_{t \to \infty}  \ddl(t) = 0.$$
\end{lemma}

\newpage
\section{Proof of Proposition \ref{prop:not_stable}}
\label{appendix:not_stable_prop}

In this appendix, we prove Proposition~\ref{prop:not_stable}. 
Let $\mathcal{C}(\LsC,\LsC)$ be a Nash equilibrium as given in  Proposition~\ref{prop:not_stable}, and we have that
$$L^*_C <  \left ( \frac{f_0}{a} - \frac{c}{a g_0} \right ).$$
We then show that for every $\delta > 0$, there exist initial perturbations $\ddl(0)$,  $\ddr(0)$,  $\dsl(0)$, and $\dsr(0)$, such that
$$ 0 < |\ddl(0)|,|\ddr(0)|,|\dsl(0)|,|\dsr(0)| < \delta,$$
such that
$$\lim_{t \to \infty} \ddl(t) > 0$$
and
$$\lim_{t \to \infty} \ddr(t) > 0,$$
as well as
$$\lim_{t \to \infty} \dsl(t) > 0$$
and
$$\lim_{t \to \infty} \dsr(t) > 0.$$
Using Definition~\ref{def:stable}~and~\ref{def:neutral-stable}, these results imply that  $\mathcal{C}(\LsC,\LsC)$ is neither a stable, or a neutral-stable, Nash equilibrium.

Let $\delta > 0$ be such that
$$  2\LsC + \delta < 2 \left ( \frac{f_0}{a} - \frac{c}{a g_0} \right ). $$

Let the constants $K$ and $M$ be as given in Lemma~\ref{lemma:dt_no_gap}. As by assumption we have that
$$L^*_C <  \left ( \frac{f_0}{a} - \frac{c}{a g_0} \right ),$$
it follows that
$$K = E_p E_q \Bsbl -2c +2f_0 g_0 - 2\LsC a g_0 \Bsbr > 0.$$
Furthermore, by definition we have that 
$$M= E_p E_q \Bsbl  2\LsC a g_0 \Bsbr > 0.$$
Using these definitions,  let $\epsilon_d$, $ 0 < \epsilon_d < \delta$,  and $\epsilon_s$,  $ 0 < \epsilon_s < \delta$, be such that
\begin{equation}\label{eq:epsilon}
  \Bsbl K \epsilon_s  - M \epsilon_d \Bsbr > 0.
\end{equation}  
Note that such a $\epsilon_d$ and $\epsilon_s$ can always be obtained by making $\epsilon_d$ small enough.

Using these definitions,  let the initial conditions 
$$ \ddl(0) = \ddr(0) = \delta_{d}(0)$$
and
$$ \dsl(0) = \dsr(0) = \delta_{s}(0),$$
be such that
$$0 < \epsilon_d <\delta_d(0) < \delta$$
and
$$0 < \epsilon_s < \delta_s(0) < \delta.$$

Under these initial conditions, we have by Lemma~\ref{lemma:dt_no_gap} that
$$\ddl(t) = \ddr(t) =  \delta_{d}(t), \qquad t \geq 0,$$
and
$$\dsl(t) = \dsr(t) =  \delta_{s}(t), \qquad t \geq 0,$$
as well as
\begin{equation}\label{eq:dt_d}
  \frac{d \ddl(t)}{\delta t} = \frac{d \ddr(t)}{\delta t} =  \frac{d \delta_{d}(t)}{\delta t} =  K \delta_s(t)  - M \delta_d(t) , \qquad t \geq 0,
\end{equation}  
and
\begin{equation}\label{eq:dt_s}
  \frac{d \dsl(t)}{\delta t} = \frac{d \dsr(t)}{\delta t} =  \frac{d \delta_{s}(t)}{\delta t} =  K \delta_d(t), \qquad t \geq 0.
\end{equation}  

We then  prove the result of the  proposition as follows. 
Suppose that
$$\lim_{t \to \infty} \delta_d(t) = 0.$$
As by by assumption we have that
$$\delta_d(0) > \epsilon_d$$
and by  Lemma~\ref{lemma:dl_cont} the function $\delta_d(t)$ is continuous in $t$ for $t \geq 0$, we then have for this case that there exists a time $t_1 > 0$ such that 
$$\delta_d(t_1) = \epsilon_d$$
and
$$\delta_d(t) >  \epsilon_d, \qquad 0 \leq t < t_1,$$
as well as
\begin{equation}\label{eq:prop1_dd}
  \frac{ d \delta_{d}(t_1)}{dt} < 0.
\end{equation}  
In the following we show that this cannot be the case.

Note that the following is true for this time $t_1$.
As we have that
$$ \delta_d(t) \geq \epsilon_d > 0, \qquad 0 \leq t \leq t_1,$$
it follows that from Eq.~\eqref{eq:dt_s} that
$$ \frac{ d \delta_{s}(t)}{dt} > 0, \qquad 0 \leq t \leq t_1,$$
and we obtain that
$$\delta_s(t_1) > \delta_s(0) > \epsilon_s.$$

Using this result, it follows from Eq.~\eqref{eq:epsilon}~and~Eq.~\eqref{eq:dt_d} that
$$\frac{ d \delta_{d}(t_1)}{dt} > 0.$$
However this result contradicts Eq.~\ref{eq:prop1_dd} which states that if
$$\lim_{t \to \infty} \delta_d(t) = 0,$$
there exists a time $t_1 > 0$ such that
$$\delta_d(t_1) = \epsilon_d$$
and
$$ \frac{ d \delta_{d}(t_1)}{dt} < 0.$$
It then follows that  we have that
$$ \delta_d(t) \geq  \epsilon_d, \qquad t \geq 0,$$
and
$$\lim_{t \to \infty} \delta_d(t) \geq \epsilon_d.$$

Moreover, for time $t \geq 0$ such that
$$\delta_d(t) \geq \epsilon_d > 0$$
we have that
$$\frac{ d \delta_{s}(t)}{dt} > 0.$$
Using the above result that
$$ \delta_d(t) \geq  \epsilon_d, \qquad t \geq 0,$$
it then follows that
$$ \delta_s(t) \geq  \delta_s(0) > \epsilon_s > 0, \qquad t \geq 0,$$
and
$$\lim_{t \to \infty} \delta_s(t) \geq \delta_s(0) > \epsilon_s > 0.$$
This establishes the result of the proposition.

\newpage
\section{Proof of Proposition \ref{prop:stable}}
\label{appendix:stable_prop}

Let $\mathcal{C}(\LsC,\lsd)$ be a Nash equilibrium with
$$\lsd = \left ( \frac{f_0}{a} - \frac{c}{a g_0} \right ) < \LsC$$
as given in Proposition~\ref{prop:stable}.
 We then show that there exists a  $\delta > 0$, such that if for initial perturbations $\ddl(0)$,  $\ddr(0)$,  $\dsl(0)$, and $\dsr(0)$, we have 
$$ 0 < |\ddl(0)|,|\ddr(0)|,|\dsl(0)|,|\dsr(0)| < \delta,$$
then we have that
$$\lim_{t \to \infty} \ddl(t) = \lim_{t \to \infty} \ddr(t) = 0.$$
Using Lemma~\ref{lemma:neutral}, this result implies that
$$ \lim_{t \to \infty} U_{C_1(t)}^{(d)}(y) =  \lim_{t \to \infty}  U_{C_2(t)}^{(d)}(y) > 0, \qquad y \in C^1_d \cup  C^2_d,$$
and by Definition~\ref{def:neutral-stable} $\mathcal{C}(\LsC,\lsd)$ is a neutral-stable Nash equilibrium.

To prove that there exists a  $\delta > 0$, such that if for initial perturbations $\ddl(0)$,  $\ddr(0)$,  $\dsl(0)$, and $\dsr(0)$, we have 
$$ 0 < |\ddl(0)|,|\ddr(0)|,|\dsl(0)|,|\dsr(0)| < \delta,$$
then we have that
$$\lim_{t \to \infty} \ddl(t) = \lim_{t \to \infty} \ddr(t) = 0,$$
we proceed as follows. 
Let
$$B_ 0 = \frac{E_p E_q \LsC a g_0}{2},$$
and let  $\delta$ be such that
$$0 < \delta < \min \left
\{ 1, \frac{L^*_c - \lsd}{2},
\frac{1}{2} \left [ \frac{f_0}{a} - \lsd \right ],
\frac{\LsC}{4},
\frac{\LsC B_0 }{ E_p E_q a g_0 \delta }
\right \}.$$
Note that for this definition of $\delta$ and for
$$|\ddl(0)| <  \delta,$$
we have that
$$-\LsC + \lsd + \ddl(0) \leq - \frac{L^*_c - \lsd}{2},$$
and
\begin{equation}\label{eq:ddl_LsC_lsd}
  \ddl(0) < L^*_c - \lsd.
\end{equation}  
Using the same argument, we obtain that
\begin{equation}\label{eq:ddr_LsC}
  - \Bbl L^*_c - \lsd \Bbr <  \ddr(0).
\end{equation}  

Using Lemma~\ref{lemma:ddl_0_1}~-~\ref{lemma:ddr_0_2} it then follows that if we have that
$$  |\dsr(t)| < \frac{\LsC}{2}, \qquad t \geq 0,$$
and
$$  |\dsr(t)| < \frac{\LsC}{2}, \qquad t \geq 0,$$
then we obtain that
$$\lim_{t \to \infty} \ddl(t) = \lim_{t \to \infty} \ddr(t) = 0.$$

Therefore in order to prove the result, it remains to show that we indeed have that
$$  |\dsl(t)| < \frac{\LsC}{2}, \qquad t \geq 0$$
and
$$  |\dsr(t)| < \frac{\LsC}{2}, \qquad t \geq 0.$$

We first consider the perturbation $\dsl(t)$. From Eq.~\eqref{eq:ddl_LsC_lsd}~and~\eqref{eq:ddr_LsC}, we have that
$$ \ddl(0) < L^*_c - \lsd$$
and
$$ - \Bbl L^*_c - \lsd \Bbr <  \ddr(0),$$
and by construction we have that
$$|\ddl(0)|, |\ddr(0)| <  \delta.$$
Using Lemma~\ref{lemma:ddl_0_1}~-~\ref{lemma:ddr_0_2}, we then obtain  that if
$$  |\dsl(t)|, |\dsr(t)| < \frac{\LsC}{2}, \qquad t \geq 0,$$
then we have that
$$ |\ddl(t)|, |\ddr(t)| <  \delta  e^{- B_0 t}, \qquad t \geq 0.$$
Combining this result with Lemma~\ref{lemma:dt_s_gap}, we obtain that if
$$  |\dsl(t)|, |\dsr(t)| < \frac{\LsC}{2}, \qquad t \geq 0,$$
then we have that
$$  \left |\frac{\dsl(t)}{\delta t} \right | <
E_p E_q \Bsbl \frac{a g_0}{2}\Bsbr \dl^2 e^{- 2B_0 t}.$$
As by construction we have that
$$ \delta < 1,$$
it follows that
$$ \delta^2 < \delta,$$
and  we obtain in this case that
$$  \left |\frac{\dsl(t)}{\delta t} \right | <
E_p E_q \Bsbl \frac{a g_0}{2}\Bsbr\delta e^{- 2B_0 t}.$$
Using this result, we obtain that the change in perturbation $\dsl(t)$ is bounded by
$$ \left | \dsl(t=0) - \dsl(t) \right |
<
\frac{E_p E_q a g_0 \delta  }{4 B_0} \qquad t \geq 0.$$

As by construction we have that
$$ | \dsl(t=0)| < \dl$$
and
$$ \delta < \min
\left \{ \frac{\LsC}{4}, \frac{\LsC B_0 }{ E_p E_q a g_0 \delta } \right \},$$
it then follows that we indeed have that
$$ |\dsl(t)| <  \frac{\LsC}{2}, \qquad t \geq 0.$$
Using the same argument, we can show that we also have that
$$ |\dsr(t)| <  \frac{\LsC}{ 2 }, \qquad t \geq 0.$$
This completes the proof to show that
$$\lim_{t \to \infty} \ddl(t) = \lim_{t \to \infty} \ddr(t) = 0.$$

\end{document}